\documentclass[a4paper,10pt,sort&compress]{amsart}
\usepackage[utf8]{inputenc}
\usepackage[english]{babel}

\usepackage{amsmath,amsfonts,amsthm,xcolor,amssymb}
\usepackage{url}
\usepackage{cleveref}
\crefname{equation}{}{}

\usepackage{lineno}
\numberwithin{equation}{section}

\newcommand{\RR}{\mathbb{R}}
\newcommand{\CC}{\mathbb{C}}
\newcommand{\kk}{\Bbbk}
\newcommand{\St}{\mathrm{St}}
\newcommand{\Skew}{\mathrm{Skew}}
\newcommand{\Gr}{\mathrm{Gr}}
\newcommand{\Tang}{\mathrm{T}}
\newcommand{\deriv}{\mathrm{d}}
\newcommand{\diag}{\mathrm{diag}}
\newcommand{\rank}{\mathrm{rank}}
\newcommand{\vect}[1]{\mathbf{#1}}

\newtheorem{lemma}{Lemma}
\newtheorem{theorem}{Theorem}
\newtheorem{corollary}{Corollary}
\newtheorem{proposition}{Proposition}
\theoremstyle{definition}

\newtheorem{remark}{Remark}

\title[Condition of singular subspaces revisited]{The condition number of singular subspaces, revisited}
\author{Nick Vannieuwenhoven}
\thanks{KU Leuven, Department of Computer Science, Celestijnenlaan 200A, B-3001 Leuven, Belgium. \texttt{nick.vannieuwenhoven@kuleuven.be}\\
This research was partially supported by the Thematic Research Programme \emph{``Tensors: geometry, complexity and quantum entanglement,''} University of Warsaw, Excellence Initiative -- Research University and the Simons Foundation Award No. 663281 granted to the Institute of Mathematics of the Polish Academy of Sciences for the years 2021-2023.}

 %ill-posedness and regularization

\begin{document}

\begin{abstract}
I revisit the condition number of computing left and right singular subspaces from [J.-G. Sun, \textit{Perturbation analysis of singular subspaces and deflating subspaces}, Numer. Math. 73(2), pp. 235--263, 1996]. For real and complex matrices, I present an alternative computation of this condition number in the Euclidean distance on the input space of matrices and the chordal, Grassmann, and Procrustes distances on the output Grassmannian manifold of linear subspaces.
Up to a small factor, this condition number equals the inverse minimum singular value gap between the singular values corresponding to the selected singular subspace and those not selected.
\end{abstract}

\maketitle 

\vspace{-10pt}
\section{Introduction}

All real or complex $m \times n$ matrices $A \in \kk^{m\times n}$ ($\kk=\RR$ or $\CC$) admit a \emph{full singular value decomposition} (SVD) \cite[Theorem 2.6.3]{HJ2013}:
\begin{linenomath}
\begin{align*}
 A = U \widehat{\Sigma} V^H, \quad 
 \widehat{\Sigma} = \begin{bmatrix} \Sigma & 0_{r\times (n-r)} \\ 0_{(m-r)\times r} & 0_{(m-r)\times(n-r)} \end{bmatrix},\quad 
 r = \rank(A),
\end{align*}
\end{linenomath} 
where $0_{p\times q}$ is a $p\times q$ zero matrix, $\Sigma$ is a real diagonal matrix $\Sigma = \diag(\sigma_1, \sigma_2, \ldots, \sigma_{r})$ that contains the nonnegative \emph{singular values} $\sigma_1 \ge \cdots \ge \sigma_r > 0$, $U \in \kk^{m \times m}$ and $V \in \kk^{n \times n}$ are unitary matrices ($U^H U = U U^H = I_m$ and $V^H V = V V^H = I_n$ equal an identity matrix), and $X^H$ denotes the conjugate transpose of $X$. 
The matrices $U$ and $V$ represent special linear subspaces of $\kk^m$ and $\kk^n$, respectively.
Indeed, denote the distinct nonnegative singular values by $\sigma_1' \ge \cdots \ge \sigma_s' > 0$ and let $d_i \ge 1$ be the multiplicity of $\sigma_i'$.
Then, we have unique orthogonal decompositions
\begin{linenomath}
\begin{align*}
 \kk^{m} 
 = \mathrm{U}_{\sigma_1'} \oplus \mathrm{U}_{\sigma_2'} \oplus \cdots \oplus \mathrm{U}_{\sigma_s'} \oplus \mathrm{U}_{0}, \quad\text{ and }\quad
 \kk^n 
 = \mathrm{V}_{\sigma_1'} \oplus \mathrm{V}_{\sigma_2'} \oplus \cdots \oplus \mathrm{V}_{\sigma_s'} \oplus \mathrm{V}_{0},
\end{align*}
\end{linenomath}
induced by the \textit{invariant subspaces} of the Hermitian matrices $A A^H$ and $A^H A$, respectively; see \cite[Chapter 2.6]{HJ2013}.
Herein, $\mathrm{U}_{\sigma_i'}$ (respectively, $\mathrm{V}_{\sigma_i'}$) denotes the $d_i$-dimensional \textit{left} (respectively, \textit{right}) \textit{singular subspace associated with $\sigma_i'$}, which is spanned by the columns of $U$ (respectively, $V$) that correspond to this singular value $\sigma_i'$. The last spaces in the decompositions, $\mathrm{U}_{0}$ and $\mathrm{V}_{0}$, of dimensions $m-r$ and $n-r$ are the \textit{cokernel} and \textit{kernel} of $A$, respectively.
More generally, we define the left and right singular subspace associated with the distinct singular values $S = \{ \sigma_{j_1}',\ldots,\sigma_{j_k}' \} \subset \{ \sigma_1', \ldots, \sigma_s', 0 \}$ as
\[
 \mathrm{U}_{S} := \mathrm{U}_{\sigma_{j_1}'} \oplus \cdots \oplus \mathrm{U}_{\sigma_{j_k}'} 
 \quad\text{and}\quad 
 \mathrm{V}_{S} := \mathrm{V}_{\sigma_{j_1}'} \oplus \cdots \oplus \mathrm{V}_{\sigma_{j_k}'}.
\]
Note that the right singular subspaces of $A$ are the left singular subspaces of $A^H$.

The sensitivity of singular subspaces is a classic topic in numerical analysis, discussed in monographs by Wilkinson \cite{Wilkinson1965}, Kato \cite{Kato1976}, and Stewart and Sun \cite{SS1990}.
One way to measure the sensitivity of a computational problem is through its condition number.
Rice \cite{Rice1966} gave a precise, general definition of the \emph{condition number} of a map between \emph{metric spaces} $f : X \to Y$, namely
\begin{linenomath}
\begin{align} \label{eqn_rice_cond}
 \kappa[f](x) := \lim_{\epsilon\to0} \sup_{\substack{x \in X,\\ d_X(x,x') \le \epsilon}} \frac{d_Y(f(x), f(x'))}{d_X(x,x')},
\end{align}
\end{linenomath}
where $d_X : X \times X \to \RR$ and $d_Y : Y \times Y \to \RR$ are distances on $X$ and $Y$, respectively. This definition implies the first-order sharp perturbation bound
\[
 d_Y( f(x'), f(x) ) \le \kappa[f](x) \cdot d_X(x',x) + o( d_X(x',x') ),
\]
which means that for all $x\in X$, there exist $x' \in X$ close to $x$, so that the inequality is an equality.
This is the essence of Rice's condition number \cite{Rice1966}, which distinguishes it from other condition-like quantities that provide only a data-dependent upper bound on $d_Y(f(x'),f(x))$ without satisfying first-order sharpness.

In pioneering work, Sun essentially computed a nice expression of Rice's condition number \cref{eqn_rice_cond} for singular subspaces in \cite[Equations (2.2.13) and (2.2.14)]{Sun1996}. 
Sun's \cite{Sun1996} approach uses the implicit function theorem to show that
for all curves $A(t) = A + t \dot{A} \subset \CC^{m\times n}$, there exists a real-analytic curve $B_{\dot{A}}(t) \subset \CC^{m\times k}$, $t\in\RR$, of \emph{arbitrary} bases that passes through a chosen \emph{orthonormal} basis $B$ of a selected $k$-dimensional singular subspace $\mathrm{U}_S$ of $A$ at $t=0$. 
Section 2.2 in \cite{Sun1996} then computes 
\begin{linenomath}
\begin{align*}
\kappa
= \lim_{\epsilon\to0}\; \sup_{\dot{A} \in \CC^{m\times n}} \frac{1}{\epsilon} \Vert B_{\dot{A}}(\epsilon) - B \Vert_F 
= \lim_{\epsilon\to0} \sup_{\dot{A} \in \CC^{m\times n}} \frac{\| \tan \boldsymbol{\theta}(\epsilon) \|}{\| A - A(\epsilon) \|_F},
\end{align*}
\end{linenomath}
where $\Vert\cdot\Vert_F$ is the Frobenius norm, and $\boldsymbol{\theta}(\epsilon)$ is the vector of principal angles between the column spans of $B$ and $B_{\dot{A}}(\epsilon)$, culminating in a closed expression for $\kappa$ in \cite[Equations (2.2.13) and (2.2.14)]{Sun1996}.
Sun considered it ``natural to regard [$\kappa$] as [a] condition number of singular subspaces'' \cite[p. 245]{Sun1996}. 
I suspect this phrasing was used by Sun because $\| \tan\boldsymbol{\theta} \|$ does not define a \emph{distance} on the space of linear subspaces, so it is not immediately obvious that the computed number coincides with Rice's definition from three decades earlier.
Nevertheless, a more precise statement is true: since $(1 + o(1)) \|\tan\boldsymbol{\theta}(\epsilon)\| =\|\sin\boldsymbol{\theta}(\epsilon)\|$ and the latter is the so-called \emph{chordal distance} on the space of linear subspaces \cite{Wong1967,YL2016,DD2013}, the computed $\kappa$ is a Rice condition number with respect to the Euclidean distance on the space of matrices and the chordal distance on the space of linear subspaces.

This paper presents an alternative, expository proof of the above condition number, both for complex and real matrices, by computing the spectral norm of a linear map between the vector space of $m\times n$ matrices and the space of linear subspaces, directly building on Rice's results \cite{Rice1966}. I believe this approach illustrates the elegance and utility of the geometric framework of condition \cite{BCSS1998,BC2013} in another classic linear algebra problem. The new computation also includes the case where the cokernel $\mathrm{U}_0$ is included in the selected left singular subspace, which was not covered in \cite{Sun1996}. I present additional results beyond \cite{Sun1996} that highlight the relations among the condition numbers of (i) left and right singular subspaces (\cref{prop_only_left}), (ii) the orthogonal complement space and the original space (\cref{prop_complementary}), and (iii) real and complex perturbations of real matrices (\cref{cor_field_extension}). A worst perturbation direction is explicitly stated as \cref{thm_secondary}.

\subsection*{Notation} 
For convenience, vectors are denoted by bold lowercase letters ($\vect{e}, \vect{u}$), matrices by uppercase letters ($A$, $U$, $V$), linear spaces by upright uppercase letters ($\mathrm{U}$), and nonlinear spaces by calligraphic uppercase letters ($\mathcal{U}, \mathcal{V}$). The cardinality of a finite set $S$ is denoted by $\sharp S$. $\kk$ refers to either the field of complex ($\CC$) or real ($\RR$) numbers. The transpose of $A\in\kk^{m\times n}$ is denoted by $A^T$ and the conjugate transpose by $A^H$. The Frobenius norm of $A$ is $\|A\|_F$, and the standard Euclidean norm of a vector $\vect{v}$ is denoted by $\|\vect{v}\|$. The vector of appropriate length that is zero everywhere except at position $i$ where it is $1$, is denoted by $\vect{e}_i$.

\subsection*{Acknowledgments}
I thank Nick Trefethen for the interesting discussions during his visit in Leuven and inquiring about the condition number of singular subspaces.

I thank two reviewers for their suggestions that led to several clarifications and to including (i) a more detailed treatment of left versus right singular subspaces, (ii) the case of complex matrices, and (iii) a worst perturbation.

Part of this research was conducted while I participated in the \emph{Algebraic Geometry with Applications to Tensors and Secants} (AGATES) thematic semester in Warsaw, Poland.

\section{Statement of the results} \label{sec_results}
Left and right subspaces of $A$ can be identified through a proper choice of indices
\begin{linenomath}
\begin{align} \label{eqn_selection}
\pi \subset \{ 1, 2, \ldots, m \}, 
\quad\text{respectively,}\quad
\rho \subset \{ 1, 2, \ldots, n \},
\end{align}
\end{linenomath}
where the cardinalities of $\pi$ and $\rho$ are respectively $\sharp \pi = k$ and $\sharp \rho = \ell$. Then, we consider the following \textit{linear subspaces}
\[
\mathrm{U}_\pi := \mathrm{span}( \vect{u}_{\pi_1}, \ldots, \vect{u}_{\pi_k} )
\quad\text{and}\quad
\mathrm{V}_\rho := \mathrm{span}( \vect{v}_{\rho_1}, \ldots, \vect{v}_{\rho_\ell} ),
\]
spanned by the columns of $U$ (respectively, $V$) in a full SVD of $A = U \Sigma V^H$ at the positions in $\pi$ (respectively, $\rho$). It is an immediate consequence of \cite[Theorem 2.6.5]{HJ2013} that $\mathrm{U}_\pi$ (respectively, $\mathrm{V}_\rho$) is well defined as a function of $A$ if and only if $\mathrm{U}_\pi$ (respectively, $\mathrm{V}_\rho$) is a left (respectively, right) singular subspace.
Therefore, I define for fixed $\pi$ and $\rho$ the following subsets:
\begin{linenomath}
\begin{align*}
\mathcal{U}_\pi &= \{ A \in \kk^{m\times n} \mid \mathrm{U}_\pi \text{ is a left singular subspace of } A \},\\
\mathcal{V}_\rho &= \{ A \in \kk^{m\times n} \mid \mathrm{V}_\rho \text{ is a right singular subspace of } A \}.
\end{align*}
\end{linenomath}
Essentially, this means that if $i \in \pi$ (respectively, $i\in\rho$) and the singular value $\sigma_i$ has a multiplicity higher than one, then all the left (respectively, right) singular vectors corresponding to it in $U$ (respectively, $V$) must be selected by $\pi$ (respectively, $\rho$); for the correct interpretation of this statement it is understood that $\sigma_{r+1}=\cdots=\sigma_{m}=0$ (respectively, $\sigma_{r+1}=\cdots=\sigma_n=0$), where $r$ is the rank of $A$.

For all matrices in $\mathcal{U}_\pi$, $\mathrm{U}_\pi$ is a left singular subspace \emph{of constant dimension equal to the cardinality of $\pi$}, i.e., $\dim \mathrm{U}_\pi = \sharp \pi$. The analogous statement holds for $\mathcal{V}_\rho$.
The set of all $d$-dimensional linear subspaces of $\kk^p$ forms a geometric object called the \textit{Grassmannian} $\Gr(d,\kk^p)$. This \textit{smooth manifold} \cite{Lee2013} can be realized, among other interpretations \cite{EAS1998,LLY2021}, as the embedded submanifold of rank-$d$ orthogonal projectors in $\kk^{p\times p}$ \cite{MS1985}:
\[
 \Gr(d, \kk^p) 
 = \{ Q Q^H \mid Q \in \kk^{p\times d},\, Q^H Q = I_d \} \subset \kk^{p\times p}.
\]

With the previous definitions in place, we now have well-defined maps 
\begin{linenomath}
\begin{align} \label{eqn_def_map}
\mathcal{L}_\pi^\kk : \mathcal{U}_\pi \to \Gr(k,\kk^m),\; A \mapsto \mathrm{U}_\pi, \quad\text{ and }\quad
\mathcal{R}_\rho^\kk : \mathcal{V}_\rho \to \Gr(\ell,\kk^n),\; A \mapsto \mathrm{V}_\rho,
\end{align}
\end{linenomath}  
which take a matrix and map it to the left and right singular subspace $\mathrm{U}_\pi$ and $\mathrm{V}_\rho$, respectively.
To compute the condition numbers of $\mathcal{L}_\pi^\kk$ and $\mathcal{R}_\rho^\kk$, we need distances on their domains and codomains.
I equip the domains with the Euclidean distance
\begin{equation}\label{eqn_eucl_metric}
d_{\kk^{m\times n}}(A,A') = \Vert A - A' \Vert_F = \sqrt{\sum_{i=1}^m \sum_{j=1}^n |a_{ij}-a_{ij}'|^2}
\end{equation}
from the ambient $\kk^{m\times n}$. For the Grassmannian $\Gr(d,\kk^p)$, several distances were considered in the literature \cite[section 12.3]{DD2013}. All unitarily-invariant distances on Grassmannians are functions of the \textit{principal angles} between $P=Q Q^H \in \Gr(d,\kk^p)$ and $P' = Q' Q'^{H} \in \Gr(d,\kk^p)$ \cite{Wong1967,YL2016}. These principal angles are
\[
 \theta_i = \cos^{-1}( \sigma_i(Q^H Q') ), \quad i=1, \ldots, d,
\] 
where $\sigma_i$ is the $i$th largest singular value of its argument. In terms of the vector of principal angles $\boldsymbol{\theta}=(\theta_1,\ldots,\theta_d)$, we will consider the following three distances:
\begin{linenomath}
\begin{align}
\nonumber d_{\Gr}^c(P,P') &:= \Vert \sin \boldsymbol{\theta} \Vert = \frac{1}{\sqrt{2}} \Vert P - P' \Vert_F, && \text{(chordal)},\\
\label{eqn_distance} d_{\Gr}^G(P,P') &:= \Vert \boldsymbol{\theta} \Vert,  && \text{(Grassmann)},\\
\nonumber d_{\Gr}^P(P,P') &:= \Vert 2 \sin (\boldsymbol{\theta}/2) \Vert = \min_{X,X' \in \mathcal{O}_k} \Vert Q X - Q' X' \Vert_F, && \text{(Procrustes)},
\end{align}
\end{linenomath}
where $\sin$ is applied elementwise and $\mathcal{O}_k$ is the unitary group of $k\times k$ unitary matrices.\footnote{The equalities are easily obtained by exploiting trigonometric identities, the fact that the Frobenius norm is induced by the Frobenius inner product, and that the trace is a function of the eigenvalues so $\mathrm{tr}(BAB^{-1})=\mathrm{tr}(A)$ for all valid $A,B$.} These distances are \emph{asymptotically} isometric \cite[section 4.3]{EAS1998}:
\[
 d_{\Gr}^c(P,P') = (1 + o(\|\boldsymbol{\theta}\|)) \cdot d_{\Gr}^G(P,P') = (1 + o(\|\boldsymbol{\theta}\|)) \cdot d_{\Gr}^P(P,P').
\]
It follows from \cref{eqn_rice_cond} that these distances can be used interchangeably without altering the condition number.

Before stating the main theorem, the following observation allows us to reduce to the case of left singular subspaces.

\begin{lemma}[Reduction to left singular subspaces]\label{prop_only_left}
If we select the chordal, Grassmann, and Procrustes distances from \cref{eqn_distance} on $\Gr(\ell,\kk^n)$ and $\Gr(k,\kk^m)$, and the Euclidean distance \cref{eqn_eucl_metric} on $\kk^{m\times n}$ and $\kk^{n\times m}$, then the following equality of condition numbers of left and right singular subspaces holds:
\[
\kappa[\mathcal{R}_\rho^\kk](A) = \kappa[\mathcal{L}_\rho^\kk](A^H).
\]
\end{lemma}

Hence, it suffices to determine the condition number of the map $\mathcal{L}_\pi^\kk$ that takes a matrix to a left singular subspace.
This condition number is the first main result.

\begin{theorem}[Condition of left singular subspaces]\label{thm_main}
Let $A \in \kk^{m \times n}$ be any matrix over $\kk=\RR$ or $\CC$. Let $\pi$ be as in \cref{eqn_selection} and let $\pi^c := \{1,\ldots,m\}\setminus\pi$. Let
\(
\sigma_{r+1}=\cdots=\sigma_m=0,
\)
where $r$ is the rank of $A$. 
Then, the condition number of $\mathcal{L}_\pi^\kk$ from \cref{eqn_def_map}, which takes $A$ to its left singular subspace $\mathrm{U}_{\pi}$, is
\begin{linenomath}
\begin{align}\label{eqn_cond_expression}
 \kappa_\pi^\kk(A) 
 :=\begin{cases}
   \kappa[\mathcal{L}^\kk_\pi](A) & \text{if } A \in \mathcal{U}_\pi \\ 
   \infty & \text{if } A \not\in \mathcal{U}_\pi
  \end{cases}
= \max_{\substack{i\in\pi,\\ j\in\pi^c}}\; \frac{1}{\vert\sigma_i-\sigma_j\vert} \cdot \sqrt{\frac{\sigma_i^2 +\sigma_j^2}{(\sigma_i + \sigma_j)^2}}
\end{align}
\end{linenomath}
with respect to the Euclidean distance \cref{eqn_eucl_metric} on the domain and the chordal, Grassmann, and Procrustes distances from \cref{eqn_distance} on the codomain. If $\pi=\emptyset$ or $\pi^c=\emptyset$, then the right-hand side is understood to evaluate to zero.
\end{theorem}

\begin{remark}\label{rem_bounded}
The condition number $\kappa_\pi^\kk$ is the maximum \emph{inverse singular value gap} between the singular values that are selected by $\pi$ and the ones that are not selected (including $\sigma_{r+1}=\cdots=\sigma_m=0$), multiplied by a scaling factor that satisfies
\[
 \frac{1}{\sqrt{2}} \le \sqrt{\frac{\sigma_i^2 +\sigma_j^2}{(\sigma_i + \sigma_j)^2}} \le 1;
\]
hence, it is rather unimportant.
\end{remark}

The theorem entails a few interesting facts that are highlighted next. First, it implies a sharp error bound for infinitesimal errors.

\begin{corollary}[Error bound]
We have the following first-order sharp error bound
\begin{linenomath}
\begin{equation*}
d_\Gr^\star\bigl( \mathcal{L}_\pi^\kk(A), \mathcal{L}_\pi^\kk(A') \bigr) \le \kappa[\mathcal{L}_\pi^\kk](A) \cdot \| A - A' \|_F \cdot (1 + o(1)),
\end{equation*}
\end{linenomath}
where $\star \in \{c, G, P\}$ is the chordal, Grassmann, or Procrustes distance from \cref{eqn_distance}.
\end{corollary}

Second, the formula for $\kappa_\pi^\kk$ is symmetric in the roles of $\pi$ and $\pi^c$. This observation is no coincidence and will be exploited to simplify the proof of \cref{thm_main}.

\begin{lemma}[Condition of the complementary subspace] \label{prop_complementary}
We have 
\[
 \kappa[\mathcal{L}_\pi^\kk](A) = \kappa[\mathcal{L}_{\pi^c}^\kk](A).
\]
That is, the condition number of the orthogonal complement $\mathrm{U}_\pi^\perp$ of a selected left singular subspace $\mathrm{U}_\pi$ is the same as the condition number of $\mathrm{U}_\pi$ itself.
\end{lemma}

Third, the condition number $\kappa_\pi^\kk$ is a continuous function of $A$ on the extended real line.
We state this implication as the next result.

\begin{corollary}[Continuity]
The condition number $\kappa_\pi^\kk : \kk^{m\times n} \to \RR\cup\{\infty\}$ is a continuous map. 
In particular, 
\[
\lim_{A \to \partial\mathcal{U}_\pi} \kappa[\mathcal{L}_\pi^\kk](A) = \infty,
\]
where $\partial \mathcal{U}_\pi$ is the boundary of $\mathcal{U}_\pi$ and $\mathcal{U}_\pi \cup \partial \mathcal{U}_\pi = \kk^{m \times n}$.
\end{corollary}

Fourth, \cref{thm_main} is a statement about three condition numbers. For complex matrices, there is one associated condition number, namely $\kappa[\mathcal{L}_\pi^\CC]$. However, for a real matrix $A\in\RR^{m\times n}$, we can look at two different condition numbers, namely $\kappa[\mathcal{L}_\pi^\RR](A)$ and $\kappa[\mathcal{L}_\pi^\CC](A)$. Both measure the the sensitivity of the left singular subspace of $A$, but differ in the set of allowed perturbations. The former allows arbitrary \textit{real} perturbations of $A$ in $\RR^{m\times n}$, while the latter allows arbitrary \textit{complex} perturbations of $A$ in $\CC^{m\times n}$. Since the set of perturbations is larger, a priori we have $\kappa[\mathcal{L}_\pi^\RR](A) \le \kappa[\mathcal{L}_\pi^\CC](A)$.
The following result follows immediately from the formula of the condition number in \cref{thm_main}.

\begin{corollary}[Invariance under field extension]\label{cor_field_extension}
 Let $A \in \RR^{m\times n}$ be a real matrix. Then, allowing complex perturbations of $A$ does not change the condition number of its associated singular subspace:
 \[
  \kappa[\mathcal{L}_\pi^\RR](A) = \kappa[\mathcal{L}_\pi^\CC](A).
 \]
\end{corollary}

Finally, relative condition numbers are readily obtained from \cref{thm_main} by an appropriate scaling of $\kappa_\pi^\kk(A)$. As we are measuring perturbations of $A$ in the Frobenius norm, scaling by $\Vert A\Vert_F$ seems most appropriate. Hence,
 \[
  \mu_\pi^{\kk}(A) := \kappa_\pi(A) \Vert A \Vert_F
 \]
 would be a natural relative condition number.
 The distances between subspaces in \cref{eqn_distance} already have a relative quality to them, being bounded respectively by $\sqrt{k}$, $\frac{\pi}{2}\sqrt{k}$, and $\sqrt{2k}$, where $k=\sharp\pi$. One could additionally divide $\mu_{\pi}^\kk$ by these maxima.

The following secondary result will also be established.

\begin{proposition}[Worst perturbation direction] \label{thm_secondary}
Let $A \in \kk^{m\times n}$ for $\kk=\RR$ or $\CC$ be of rank $r$. Let $A=U\Sigma V^H$ be a full SVD of $A$ with $\sigma_1 \ge \cdots \ge \sigma_r > \sigma_{r+1} = \cdots = \sigma_m = 0$.
Let $i\in\pi$ and $j\in\pi^c$ be any indices that realize the condition number in \cref{thm_main}.
Then, the perturbation
\[
\dot{A} = U (\vect{e}_j \vect{e}_i^T - \vect{e}_i \vect{e}_j^T)\Sigma V^H + \begin{cases}
2\frac{\sigma_i \sigma_j}{\sigma_i^2 + \sigma_j^2} U \Sigma (\vect{e}_i \vect{e}_j^T - \vect{e}_j \vect{e}_i^T) V^H,  & \text{if } 1 \le i,j \le r, \\
0, & \text{otherwise},
\end{cases}
\]
infinitesimally attains the condition number as $t\to0$:
 \[
d_{\Gr}^\star\bigl( \mathcal{L}_\pi^\kk(A), \mathcal{L}_\pi^\kk(A+t\dot{A}) \bigr) = \kappa[\mathcal{L}_\pi^\kk](A) \cdot \|t \dot{A}\|_F \cdot (1 + o(1))
 \]
 where $\star\in\{c,G,P\}$ is the chordal, Grassmann, or Procrustes distance from \cref{eqn_distance}.
\end{proposition}

\section{Illustration in an example}\label{sec_example}
Let us consider an example to illustrate what \cref{thm_main} does and does not state. Since the result does not depend on the singular vectors themselves, we can consider pseudodiagonal matrices without loss of generality. Consider the $6 \times 5$ real matrix
\[
 A = 
\begin{bmatrix} 
\Sigma \\
0_{1\times5}
\end{bmatrix}, \;\text{where}\quad \Sigma=\diag(4,2,1,0.99,0).
\]
Hence, $\sigma_1=4$, $\sigma_2=2$, $\sigma_3=1$, $\sigma_4=0.99$, and $\sigma_5=\sigma_6=0$.
When selecting just one singular value, so $\sharp\pi=1$, we find
\begin{linenomath}
\begin{align*}
 \kappa_{\{1\}}(A) &= 2^{-1} \sqrt{\frac{4^2+2^2}{6^2}} \approx 0.37, &&& 
 \kappa_{\{2\}}(A) &= 1^{-1} \sqrt{\frac{2^2+1^2}{3^2}} \approx 0.75, \\
 \kappa_{\{3\}}(A) &= 100 \sqrt{\frac{1^2 + .99^2}{1.99^2}} \approx 70.71, &&&
 \kappa_{\{4\}}(A) &= 100 \sqrt{\frac{1^2 + .99^2}{1.99^2}} \approx 70.71, \\
 \kappa_{\{5\}}(A) &= \infty, &&& \kappa_{\{6\}}(A) &= \infty.
\end{align*}
\end{linenomath}
I dropped the superscript because the field ($\RR$ or $\CC$) does not matter by \cref{cor_field_extension}.
The last two condition numbers are $\infty$ because the $2$-dimensional cokernel $\langle \vect{e}_5, \vect{e}_6 \rangle$ is split across the selected and the non-selected singular values $\sigma_5=0$ and $\sigma_6=0$.

We can verify these numbers numerically as follows.
The worst-case perturbation for both $\mathrm{U}_{\sigma_3}$ and $\mathrm{U}_{\sigma_4}$ is, up to scale,
\begin{linenomath}
\begin{align*}
\dot{A} 
&= (\sigma_4 \vect{e}_3 \vect{e}_4^T - \sigma_3 \vect{e}_4 \vect{e}_3^T)  + 2\frac{\sigma_3\sigma_4}{\sigma_3^2 + \sigma_4^2} (\sigma_4 \vect{e}_4 \vect{e}_3^T - \sigma_3 \vect{e}_3 \vect{e}_4^T),
\end{align*}
\end{linenomath}
by \cref{thm_secondary}.
Adding $\epsilon = 10^{-5}$ times this perturbation to $A$ only affects the $2 \times 2$ submatrix formed by the third and fourth rows and columns of $A$. Restricting to it, we find numerically using Octave that the SVD of
\[
A'_\epsilon = \begin{bmatrix}
 1 & -9.9494975\cdot10^{-8} \\ 
 -1.00499975\cdot10^{-7} & 0.99
\end{bmatrix}
\]
is given approximately by
\begin{linenomath}
\begin{multline*}
A'_\epsilon =
\begin{bmatrix}
  9.999999999500002 \cdot 10^{-1} & 9.999999998500292 \cdot 10^{-6} \\
  -9.999999998500292 \cdot 10^{-6} & 9.999999999500001 \cdot 10^{-1}
\end{bmatrix}
\cdot 
\diag(1, 0.99)\\
\cdot
\begin{bmatrix}
9.999999999500051 \cdot 10^{-1} &  -9.999494973501580 \cdot 10^{-6} \\
9.999494973501580 \cdot 10^{-6} &   9.999999999500051 \cdot 10^{-1}
\end{bmatrix}.
\end{multline*}
\end{linenomath}
The chordal distance between $\vect{e}_3$ and $(0,0,u_{11},u_{21},0,0)$ is $9.999978209007872 \cdot 10^{-6}$, where $(u_{11},u_{21})$ is the first column of the matrix of left singular vectors in the above factorization of $A_\epsilon'$. As the Frobenius distance between $A$ and $A_\epsilon'$ is only $1.414195706930316 \cdot 10^{-7}$, their fraction is approximately $70.71 \approx \kappa_{\{3\}}(A)$.

Looking at all possible condition numbers with $\sharp\pi=2$, we have 
\begin{linenomath}
\begin{align*}
 \kappa_{\{1,2\}}(A) &= \kappa_{\{2\}}(A),  &&& \kappa_{\{1,3\}}(A) &= \kappa_{\{3\}}(A), &&& \kappa_{\{1,4\}}(A) = \kappa_{\{4\}}(A), \\
 \kappa_{\{2,3\}}(A) &= \kappa_{\{3\}}(A), &&& \kappa_{\{2,4\}}(A) &= \kappa_{\{4\}}(A), \\
 \kappa_{\{3,4\}}(A) &= 0.99^{-1} \approx 1.01, \\
 \kappa_{\{5,6\}}(A) &= 0.99^{-1} \approx 1.01,
\end{align*}
\end{linenomath}
while all other combinations, which split $\sigma_5=0$ and $\sigma_6=0$, are equal to $\infty$. Note that the singular subspace $\mathrm{U}_{\{\sigma_3,\sigma_4\}}=\langle \vect{e}_3,\vect{e}_4 \rangle$ is well conditioned with $\kappa_{\{3,4\}}(A) \approx 0.75$, while the individual singular subspaces $\mathrm{U}_{\sigma_3}=\langle \vect{e}_3\rangle$ and $\mathrm{U}_{\sigma_4}=\langle \vect{e}_4 \rangle$ are not.

It is important to stress that $\kappa_{\{5,6\}}(A)$ is \emph{not} the condition number of computing the cokernel $\mathrm{U}_{\sigma_5}=\mathrm{U}_0$ of $A$. Indeed, an infinitesimal perturbation of $A$ can cause the dimension of the cokernel to drop from $2$ to $1$, e.g., when adding $\epsilon \vect{e}_5 \vect{e}_5^T$ to $A$. This means that the codomain of a map that takes an arbitrary matrix $A$ to its cokernel would have to be the \emph{double-infinite Grassmannian} \cite{YL2016}. However, such a map is necessarily discontinuous when the dimension jumps, entailing a condition number equal to $\infty$ with any of the distances in \cite{YL2016}.

What $\kappa_{\{5,6\}}$ is measuring instead is how the left singular subspace associated with the $5$th and $6$th largest singular value, i.e., the space represented by the projector $\mathcal{L}_{\{5,6\}}(A) = \vect{e}_5 \vect{e}_5^T + \vect{e}_6 \vect{e}_6^T$, is (continuously) moving as $A$ is perturbed. Note that these perturbations may increase $\sigma_5$ from $0$ to a nonzero value. For example, adding $\epsilon \vect{e}_5 \vect{e}_5^T$ to $A$ increases $\sigma_5$ from $0$ to $\epsilon$. While this would constitute a worst perturbation of the cokernel, which changes to the one-dimensional space $\vect{e}_6 \vect{e}_6^T$, the projector associated with the $5$th and $6$th largest singular values is $\mathcal{L}_{\{5,6\}}(A + \epsilon \vect{e}_5 \vect{e}_5^T) = \vect{e}_5 \vect{e}_5^T + \vect{e}_6 \vect{e}_6^T$. That is, $\vect{e}_5 \vect{e}_5^T$ is a best infinitesimal perturbation direction in which the singular subspace remains constant.
On the other hand, the worst perturbation of $\mathcal{L}_{\{5,6\}}$ is given by \cref{thm_secondary}, which in this case is up to scale equal to
\(
 \dot{A} = \sigma_4 \vect{u}_\perp \vect{e}_4^T,
\)
where $\vect{u}_\perp$ is any element in the cokernel of $A$, i.e., $\vect{u}_\perp = \alpha \vect{e}_5 + \beta \vect{e}_6$ for some $\alpha,\beta\in\RR$. Taking, for example, $\alpha=1$ and $\beta=0$, and adding $\epsilon \dot{A}$ to $A$, we find that $\Vert A - (A+\epsilon\dot{A})\Vert_F = \epsilon \sigma_4$, while $\vect{e}_4$ is perturbed to
\[
 \frac{\vect{e}_4 + \epsilon \vect{u}_\perp}{\Vert \vect{e}_4 + \epsilon \vect{u}_\perp\Vert} = \frac{1}{\sqrt{1 + \epsilon^2}} (\vect{e}_4 + \epsilon \vect{u}_\perp).
\]
Consequently, 
\[
 \mathcal{L}_{\{5,6\}}(A+\epsilon\dot{A}) = I - \vect{e}_1 \vect{e}_1^T - \vect{e}_2 \vect{e}_2^T - \vect{e}_3 \vect{e}_3^T - \frac{1}{1+\epsilon^2} (\vect{e}_4 + \epsilon \vect{u}_\perp)(\vect{e}_4 + \epsilon \vect{u}_\perp)^T.
\]
Since the chordal distance is also equal to $\frac{1}{\sqrt{2}}$ times the Frobenius distance between the projectors, we have 
\begin{linenomath}
\begin{multline*}
 \frac{1}{\sqrt{2}} \Bigl\Vert I - \sum_{i=1}^3 \vect{e}_i \vect{e}_i^T - \frac{1}{1+\epsilon^2} (\vect{e}_4 + \epsilon \vect{u}_\perp)(\vect{e}_4 + \epsilon \vect{u}_\perp)^T - (\vect{e}_5 \vect{e}_5^T + \vect{e}_6 \vect{e}_6^T) \Bigr\Vert_F
 \\= \frac{1}{\sqrt{2}} \Bigl\Vert \vect{e}_4 \vect{e}_4^T - \frac{1}{1+\epsilon^2} (\vect{e}_4 + \epsilon \vect{u}_\perp)(\vect{e}_4 + \epsilon \vect{u}_\perp)^T \Bigr\Vert_F \approx 9.999999999500001 \cdot 10^{-6}.
\end{multline*}
\end{linenomath}
Dividing by $\epsilon \sigma_4= 0.99 \cdot 10^{-5}$ yields the accurate approximation $1.010101010050505$ of the condition number $\kappa_{\{5,6\}} =  0.99^{-1} \approx 1.0101\ldots$

\section{Proof of \cref{prop_only_left,prop_complementary}} \label{sec_auxiliary}
I prove the two lemmas that simplify the computation of the condition number. \Cref{prop_only_left} stated that the condition number of a right singular subspace is the condition number of the corresponding left singular subspace of $A^H$.

\begin{proof}[Proof of \cref{prop_only_left}]
The fact that the right singular subspace $\mathcal{R}^\kk_\rho(A)$ equals the left singular subspace $\mathcal{L}^\kk_\rho(A^H)$ for $A \in \mathcal{V}_\rho$ follows from the definitions. Since conjugate transposition defines an isometry between $(\kk^{m \times n}, d_{\kk^{m\times n}})$ and $(\kk^{n \times m}, d_{\kk^{n \times m}})$, and the condition number in \cref{eqn_rice_cond} is invariant under isometries of the domain, the proof is concluded.
\end{proof}

\Cref{prop_complementary} stated that the condition number of a subspace $\mathrm{U}_\pi$ is the same as that of its orthogonal complement $\mathrm{U}_\pi^\perp = \mathrm{U}_{\pi^c}$. This is proved next.

\begin{proof}[Proof of \cref{prop_complementary}]
When $P \in \Gr(k,\RR^m)$ is a projector, then we can define its complement projector as $P^\perp = I - P \in \Gr(m-k, \RR^m)$. It follows from the definition of the chordal distance that orthogonal complementation $\perp : \Gr(k,\RR^m)\to\Gr(m-k,\RR^m)$ defines an isometry in the chordal distance. As the condition number is invariant under isometries of its codomain, the proof is concluded.
\end{proof}

\section{Proof of \cref{thm_main}: The real case} \label{sec_proof_main_thm}
This section presents the proof of \cref{thm_main} for $\kk=\RR$.
At a glance, the proof proceeds as follows in \cref{sec_sub_assumption,sec_sub_pullback,sec_sub_smooth,sec_sub_simplifying,sec_sub_eliminating,sec_sub_condition,sec_sub_thm1_equal}:
\begin{enumerate}
 \item the results of \cref{sec_auxiliary} are exploited to simplify the $\pi$'s we need to consider;
 \item the condition number \cref{eqn_rice_cond} can be computed from the derivative of $\mathcal{L}_\pi^\RR$, which reduces determining the condition number to a spectral norm computation: a norm maximization subject to a norm constraint;
 \item an explicit smooth map that locally coincides with $\mathcal{L}_\pi^\RR$ is presented and its derivative is computed;
 \item the norms appearing in the spectral norm computation are simplified;
 \item most of the variables appearing in the optimization problem are eliminated;
 \item a closed solution of the reduced optimization problem is determined; and
 \item uniqueness of continuous extensions from dense subsets is exploited to conclude that the formula obtained in steps $1$--$6$ applies for all real matrices.
\end{enumerate}

\begin{remark}
Let $\pi$ be as in \cref{eqn_selection}.
Throughout steps $1$--$6$ (i.e., \cref{sec_sub_assumption,sec_sub_pullback,sec_sub_smooth,sec_sub_simplifying,sec_sub_eliminating,sec_sub_condition}), $A \in \mathcal{U}_\pi \subset \RR^{m\times n}$ is a real matrix of rank equal to $r=\min\{m,n\}$ whose singular values are distinct. Then, in the final step, a uniqueness of extensions argument shows that the obtained formula applies for all matrices.\end{remark}

\subsection{Simplifying assumptions}\label{sec_sub_assumption}
If $i\in\pi$ such that $r+1 \le i \le m$, then a vector from the cokernel of $A$ is selected. However, since $A \in \mathcal{U}_\pi$ this must imply $\{r+1,\ldots,m\} \subset \pi$; otherwise, $\mathrm{U}_\pi$ cannot be a left singular subspace. Therefore, either $\pi \subset \{1,\ldots,r\}$ or $\pi^c \subset \{1,\ldots,r\}$. By \cref{prop_complementary}, $\kappa[\mathcal{L}_\pi^\RR](A) = \kappa[\mathcal{L}_{\pi^c}^\RR](A)$, so we can assume without loss of generality that $\pi \subset \{1,\ldots,r\}$. In fact, by relabeling the nonzero, distinct singular values, we can further assume, without loss of generality, that $\pi = \{1, \ldots, k\}$, where $k \le r$, simplifying the notation.
I also define
\[
S = \begin{bmatrix} I_{k} \\ 0 \end{bmatrix} \begin{bmatrix} I_{k} & 0 \end{bmatrix}
\]
as the orthogonal projection matrix that selects the first $k$ coordinates and zeros out the other coordinates. It is idempotent ($S^2 = S$) and symmetric ($S=S^T$).

\subsection{From a manifold to the tangent space}\label{sec_sub_pullback}
The codomain of $\mathcal{L}_\pi^\RR$ from \cref{eqn_def_map} is the Grassmannian $\Gr(k,\RR^m)$, where $k=\sharp\pi$.
Viewed as the manifold of rank-$k$ orthogonal projectors, it is an embedded submanifold of $\RR^{m \times m}$ \cite{MS1985}. Locally at $P \in \Gr(k,\RR^m)$, the manifold can be identified with the \textit{tangent space} $\Tang_P \Gr(k,\RR^m)$, a linear subspace of $\RR^{m\times m}$ with origin at $P$ that consists of all derivatives of smooth curves in $\Gr(k,\RR^m)$ passing through $P$.
The tangent spaces $\Tang_P \Gr(k,\RR^m)$ can be equipped with the \textit{Riemannian metric}
\begin{linenomath}
\begin{align*}%\label{eqn_riemannian}
g_P(\dot{P},\dot{Q}) := \frac{1}{2} \langle \dot{P}, \dot{Q} \rangle_F = \frac{1}{2} \mathrm{trace}(\dot{P}^T \dot{Q}), \quad\text{for all } \dot{P},\dot{Q} \in \Tang_P \Gr(k,\RR^m).
\end{align*}
\end{linenomath}
The structure $(\Gr(k,\RR^m),g)$ is called a \emph{Riemannian manifold} \cite{Petersen2006}.
The inner product $g_P$ induces the norm $\frac{1}{\sqrt{2}}\|\dot{P}\|_F = ( g_P(\dot{P},\dot{P}) )^{1/2}$, which, in turn, induces a \textit{Riemannian distance} on $\Gr(k,\RR^m)$ \cite{Petersen2006}.
Wong \cite{Wong1967} showed that $g$ induces the Grassmann distance $d_\Gr^G$ from \cref{eqn_distance}; see also \cite{EAS1998}.

The domain of $\mathcal{L}_\pi^\RR$ is the set $\mathcal{U}_\pi$.
The subset $\mathcal{U}\subset\RR^{m\times n}$ of full-rank matrices with distinct singular values is open; see, e.g., \cref{lem_zariski} in \cref{sec_sub_thm1_equal}. Moreover, the basic left singular subspaces of $A \in \mathcal{U}$, except for the cokernel $\mathrm{U}_0$, are all $1$-dimensional. Hence, with the above assumption that $\pi\subset\{1,\ldots,r\}$, it follows that $\mathcal{U}\subset\mathcal{U}_\pi$.
As the definition of condition in \cref{eqn_rice_cond} is local, we have $\kappa[\mathcal{L}_\pi^\RR](A) = \kappa[\mathcal{L}_\pi^\RR|_{\mathcal{U}}](A)$. Since $\mathcal{U} \subset \RR^{m\times n}$ is open, it is a smooth embedded Riemannian manifold inheriting the inner product $\langle\cdot,\cdot\rangle_F$ and distance $d_{\RR^{m\times n}}(A,A')$ from $\RR^{m\times n}$ \cite[Chapter 5]{Petersen2006}.

From the above information, it follows that the optimization problem \cref{eqn_rice_cond} is over Riemannian manifolds (i.e., nonlinear curved spaces). By the main result in \cite{Rice1966}, it can be pulled back to the tangent spaces (i.e., linear spaces).
Computing the condition number $\kappa[\mathcal{L}_\pi^\RR](A)$ thus reduces to the spectral norm computation
\begin{linenomath}
\begin{align}\label{eqn_simplified_condition}
 \kappa[\mathcal{L}_\pi^\RR](A)
 =\lim_{\epsilon\to0} \sup_{\substack{A \in \RR^{m\times n},\\ \|A-A'\|_F \le \epsilon}} \frac{d_{\Gr}^G(\mathcal{L}_\pi^\RR(A), \mathcal{L}_\pi^\RR(A'))}{\| A - A' \|_F}
 = \sup_{\dot{A}\in\RR^{m\times n}} \frac{ \tfrac{1}{\sqrt{2}}\| (\deriv_A \mathcal{L}_\pi^\RR)(\dot{A}) \|_F }{ \| \dot{A} \|_F },
\end{align}
\end{linenomath}
insofar as the \textit{derivative} of $\mathcal{L}_\pi^\RR$ \cite{Lee2013}, i.e.,
$\deriv_A \mathcal{L}_\pi^\RR : \RR^{m\times n} \to \Tang_{\mathcal{L}_\pi^\RR(A)} \Gr(k,\RR^m)$, exists.  The derivative of a differentiable map is a linear map between tangent spaces \cite{Lee2013}.

\subsection{Characterizing $\mathcal{L}_\pi^\RR$} \label{sec_sub_smooth}
Let $\mathrm{D}^r$ denote the subspace of real $r\times r$ diagonal matrices with strictly positive entries. Usually, $\mathcal{L}_\pi^\RR$ can be realized as the composition $\mathrm{P}_\pi \circ \Pi^{-1}$ of the smooth maps
\begin{linenomath}
\begin{align*}
 \Pi: \St_{m,r} \times \mathrm{D}^r \times \St_{n,r} \to \RR^{m \times n},\quad
 (U, \Sigma, V) \mapsto U\Sigma V^T,\\
 \mathrm{P}_\pi : \St_{m,r} \times \mathrm{D}^r \times \St_{n,r} \to \Gr(k,\RR^m),\quad
 (U, \Sigma, V) \mapsto U S U^T,
\end{align*}
\end{linenomath}
where $\St_{m,r} = \{ U \in \RR^{m \times r} \mid U^T U = I_r \}$ is the smooth manifold of $m\times r$ matrices with orthonormal columns \cite{Lee2013,EAS1998}.
Recall that $(U,\Sigma,V)$ is a \textit{compact SVD} of $A = U \Sigma  V^T$ if $U\in\St_{m,r}$, $\Sigma\in\mathrm{D}^r$, $V\in\St_{n,r}$.
The fiber of $\Pi$ over a rank-$r$ matrix $A$ is thus the set of compact SVDs of $A$.
The map $\mathrm{P}_\pi$ takes a compact SVD and maps it to the subspace spanned by the columns of $U$ at the positions in $\pi$.
The fiber of $\Pi$ at $A \in \mathcal{U}_\pi$ composed with $\mathrm{P}_\pi$ results in a unique point on the Grassmannian; see also the definition of $\mathcal{U}_\pi$ in \cref{sec_results}.
Consequently, the composition $\mathrm{P}_\pi \circ \Pi^{-1}$ is well defined on $\mathcal{U}_\pi$, even though $\Pi$ is never injective.

Let $A = U \Sigma  V^T$ be a compact SVD. The derivative of $\Pi$ at $(U,\Sigma ,V)$ is
\begin{linenomath}
\begin{align*}
\deriv_{(U,\Sigma ,V)} \Pi : 
\Tang_{U} \St_{m,r} \times \Tang_\Sigma  \mathrm{D}^r \times \Tang_{V} \St_{n,r} &\to \Tang_A\RR^{m \times n},\\
(\dot{U}, \dot{\Sigma }, \dot{V}) &\mapsto \dot{U} \Sigma  V^T + U \dot{\Sigma } V^T + U \Sigma  \dot{V}^T,
\end{align*}
\end{linenomath}
where $\Tang_x X$ is the tangent space at $x\in X$ of a manifold $X$.
We can decompose
\begin{linenomath}
 \begin{align}  \label{eqn_proof_decomposition}
\dot{U} = U \dot{\Lambda}_U + U_\perp \dot{X} \quad\text{and}\quad \dot{V} = V \dot{\Lambda}_V + V_\perp \dot{Y},
 \end{align}
\end{linenomath}
where $\dot{\Lambda}_U, \dot{\Lambda}_V \in \mathrm{Skew}_r$ are $r \times r$ skew-symmetric matrices, $U_\perp$ and $V_\perp$ contain orthonormal bases of the orthogonal complements of the column spans of $U$ and $V$, respectively, and $\dot{X} \in \RR^{m-r \times r}$ and $\dot{Y} \in \RR^{n-r \times r}$; see, e.g., \cite{EAS1998}.
\begin{lemma} \label{lem_injective}
$\deriv_{(U,\Sigma ,V)} \Pi$ is injective if all $\sigma_i$ in $\Sigma =\mathrm{diag}(\sigma_1,\ldots,\sigma_r)\in\mathrm{D}^r$ are distinct.
\end{lemma}
\begin{proof}
Suppose there is some $(\dot{U},\dot{\Sigma },\dot{V})$ in the kernel. Then, we must have 
\[
\deriv_{(U,\Sigma ,V)} \Pi (\dot{U},\dot{\Sigma },\dot{V}) = 
 \begin{bmatrix} U & U_\perp \end{bmatrix} \begin{bmatrix} \dot{\Lambda}_U \Sigma  + \dot{\Sigma } - \Sigma  \dot{\Lambda}_V & \Sigma \dot{Y}^T \\ \dot{X}\Sigma  & 0 \end{bmatrix} \begin{bmatrix} V & V_\perp \end{bmatrix}^T = 0.
\]
Hence, $\dot{X} = 0$, $\dot{Y} = 0$, and $\dot{\Lambda}_U \Sigma  + \dot{\Sigma } - \Sigma  \dot{\Lambda}_V = 0$. The last equation is equivalent to $\Sigma ^{-1} \dot{\Lambda}_U \Sigma  = \dot{\Lambda}_V$ and $\dot{\Sigma }=0$ because the diagonal of skew-symmetric matrices is zero. Now, $\Sigma ^{-1} \dot{\Lambda}_U \Sigma$ should be a skew-symmetric matrix because $\dot{\Lambda}_V$ is. This requires $(\Sigma ^{-1} \dot{\Lambda}_U \Sigma)_{ij} = \sigma_i^{-1} \dot{\lambda}_{ij} \sigma_j = - \sigma_j^{-1} \dot{\lambda}_{ji} \sigma_i = -(\Sigma ^{-1} \dot{\Lambda}_U \Sigma)_{ji}$, or, equivalently, $\sigma_j^2 \dot{\lambda}_{ij} = \sigma_i^2 \dot{\lambda}_{ij}$, having used that $\dot{\Lambda}_U$ is skew symmetric. As all $\sigma_i$ are distinct and nonzero, this implies $\dot{\lambda}_{ij}=0$ for all $i,j$. Hence the kernel is trivial.
\end{proof}

Since the dimensions of domain and codomain of $\Pi$ are equal, it follows from this lemma that $\Pi$ is a local diffeomorphism \cite[Proposition 4.8]{Lee2013}. Since the elements in the fiber of $\Pi$ are isolated over $\mathcal{U}$, it follows that $\Pi$ restricts to a \textit{diffeomorphism} (a bijective smooth map with smooth inverse map) $\Pi_{(U,\Sigma,V)}^{-1}$ between an open neighborhood $\mathcal{Y}$ of $(U,\Sigma,V)$ in $\St_{m,r} \times \mathrm{D}^r \times \St_{n,r}$ and an open neighborhood $\mathcal{X}$ of $A=U\Sigma V^T \in \mathcal{U}_\pi$. Hence, $\Pi_{(U,\Sigma,V)}^{-1}$ is a well-defined inverse map of $\Pi$, tracking a specific compact SVD in the fiber of $\Pi$. Hence, $\mathcal{L}_\pi^\RR\vert_\mathcal{X} = \mathrm{P}_\pi \circ \Pi_{(U,\Sigma,V)}^{-1}$ in the open submanifold $\mathcal{X}$. To determine the condition number, equality of $\mathcal{L}_\pi^\RR$ and $\mathrm{P}_\pi \circ \Pi_{(U,\Sigma,V)}^{-1}$ on $\mathcal{X}$ suffices.

\subsection{Simplifying norms} \label{sec_sub_simplifying}
As $\deriv_{(U,\Sigma ,V)}{\Pi}$ is injective, we have, by the inverse function theorem for manifolds \cite{Lee2013}, that
\[
 (\deriv_{U\Sigma V^T}{\Pi_{{(U,\Sigma,V)}}^{-1}})(\dot{U} \Sigma  V^T + U \dot{\Sigma } V^T + U \Sigma  \dot{V}^T) = (\dot{U},\dot{\Sigma },\dot{V}).
\]
This can be composed with the derivative of $\mathrm{P}_\pi$:
\begin{linenomath}
\begin{align*}
 \deriv_{(U,\Sigma,V)} \mathrm{P}_\pi : \Tang_U \St_{m,r} \times \Tang_\Sigma  \mathrm{D}^r \times \Tang_V \St_{n,r} &\to \Tang_{U S U^T} \Gr_{m,k}\\ %+ \delta(I-UU^T)
 (\dot{U},\dot{\Sigma },\dot{V}) &\mapsto \dot{U} S U^T + U S \dot{U}^T.
\end{align*}
\end{linenomath}
Plugging $\mathcal{L}_\pi^\RR|_\mathcal{X} = \mathrm{P}_\pi \circ \Pi_{(U,\Sigma,V)}^{-1}$ into \cref{eqn_simplified_condition} yields
\begin{linenomath}
\begin{align}\label{eqn_cond_sup_def}
\kappa[\mathcal{L}_\pi^\RR](A) 
\nonumber&= \sup_{\substack{\dot{A} \in \Tang_{A}\mathcal{X},\\ \Vert\dot{A}\Vert_F=1}} \frac{1}{\sqrt{2}} \left\Vert (\deriv_{({U},{\Sigma },{V})}{\mathrm{P}_\pi} ) \left( ( \deriv_{A}{\Pi_{(U,\Sigma,V)}^{-1}}) (\dot{A}) \right) \right\Vert_F\\
&= \sup_{\substack{\dot{U}\in\Tang_{U}{\mathrm{St}_{m,r}},\, \dot{\Sigma}\in\Tang_\Sigma \mathrm{D}^r,\, \dot{V}\in\Tang_{V}\St_{n,r},\\ \Vert\dot{U}\Sigma V^T + U \dot{\Sigma} V^T + U \Sigma  \dot{V}^T \Vert_F = 1}} \frac{1}{\sqrt{2}} \Vert \dot{U} S U^T + U S \dot{U}^T \Vert_F.
\end{align}
\end{linenomath}

Next, we simplify both norms in \cref{eqn_cond_sup_def}. Consider \cref{eqn_proof_decomposition} and let us partition 
\begin{equation*}
\dot{\Lambda}_U = \begin{bmatrix} \dot{B} & -\dot{E}^T \\ \dot{E} & \dot{C} \end{bmatrix}
 \quad\text{ and }\quad
 \dot{\Lambda}_V = \begin{bmatrix} \dot{B}' & -\dot{E}'^{T} \\ \dot{E}' & \dot{C}' \end{bmatrix},
\end{equation*}
where $\dot{B}, \dot{B}' \in \mathrm{Skew}_{k}$ and $\dot{C}, \dot{C}' \in \mathrm{Skew}_{r-k}$ are skew symmetric matrices, and $\dot{E}, \dot{E}' \in \RR^{r-k \times k}$ are arbitrary.
Then, on the one hand, we have
\begin{linenomath}
\begin{align}
\nonumber \Vert \dot{U} S U^T + U S \dot{U}^T \Vert_F^2
\nonumber&= \Vert U_\perp \dot{X} S U^T \Vert_F^2 + \Vert U (\dot{\Lambda}_U S + S \dot{\Lambda}_U^T) U^T \Vert_F^2 + \Vert U S \dot{X}^T U_\perp^T \Vert_F^2\\
\nonumber&= 2 \Vert \dot{X} S \Vert_F^2 + \Vert \dot{\Lambda}_U S - S \dot{\Lambda}_U \Vert_F^2 \\
\nonumber&= 2 \Vert \dot{X} S \Vert_F^2 + \left\Vert \begin{bmatrix} \dot{B} & 0 \\ \dot{E} & 0 \end{bmatrix} - \begin{bmatrix} \dot{B} & -\dot{E}^T \\ 0 & 0 \end{bmatrix} \right\Vert_F^2 \\
\label{eqn_norm1} &= 2 \Vert \dot{X} S \Vert_F^2 + 2 \Vert\dot{E}\Vert_F^2.
\end{align}
\end{linenomath}
On the other hand, we obtain 
\begin{linenomath}
\begin{align}
\nonumber\Vert\dot{U}\Sigma V^T& + U \dot{\Sigma } V^T + U \Sigma  \dot{V}^T \Vert_F^2\\
\nonumber&= \Vert U_\perp \dot{X} \Sigma V^T \Vert_F^2 + \Vert U (\dot{\Lambda}_U \Sigma + \dot{\Sigma } + \Sigma  \dot{\Lambda}_V^T) V^T \Vert_F^2 + \Vert U \Sigma  \dot{Y}^T V_\perp^T \Vert_F^2 \\
\label{eqn_exp_norm21}&= \Vert \dot{X} \Sigma  \Vert_F^2 + \Vert \dot{Y} \Sigma  \Vert_F^2 + \Vert \dot{\Sigma } \Vert_F^2 + \Vert \dot{\Lambda}_U \Sigma  - \Sigma  \dot{\Lambda}_V \Vert_F^2,
\end{align}
\end{linenomath}
where in the last step we used that $\dot{\Sigma }$ is a diagonal matrix, while the diagonal of a skew-symmetric matrix is zero, so that their inner product is zero. 
By partitioning $\Sigma = \left[\begin{smallmatrix}\Sigma_1 & 0 \\ 0 & \Sigma_2 \end{smallmatrix}\right]$ with $\Sigma_1 \in \mathrm{D}^k$, we find
\begin{linenomath}
\begin{multline}\label{eqn_exp_norm22}
 \Vert \dot{\Lambda}_U \Sigma  - \Sigma  \dot{\Lambda}_V \Vert_F^2 
 =\\ \Vert \dot{B} \Sigma_1 - \Sigma_1 \dot{B}' \Vert_F^2 + \Vert \dot{C} \Sigma_2 - \Sigma_2 \dot{C}' \Vert_F^2 + \Vert \dot{E} \Sigma_1 - \Sigma_2 \dot{E}' \Vert_F^2 + \Vert \Sigma_2 \dot{E} - \dot{E}' \Sigma_1 \Vert_F^2.
\end{multline}
\end{linenomath}

Plugging \cref{eqn_exp_norm21,eqn_exp_norm22,eqn_norm1} into \cref{eqn_cond_sup_def}, we obtain for $A=U\Sigma V^T$ the following optimization problem over Euclidean spaces:
\begin{linenomath}
\begin{align}\label{eqn_sup_complicated}
\kappa[\mathcal{L}_\pi^\RR](A) 
= \sup_{
\substack{
\dot{X}\in\RR^{m-r\times r},\; \dot{Y}\in\RR^{n-r\times r},\\
\dot{E}, \dot{E}'\in\RR^{r-k \times k},\\
\dot{B}, \dot{B}'\in \Skew_{k},\; \dot{C}, \dot{C}' \in \Skew_{r-k},\\
\dot{\Sigma} \in \Tang_{\Sigma} \mathrm{D}^r,\\
\Vert\dot{X}\Sigma \Vert_F^2 + \Vert\dot{Y}\Sigma \Vert_F^2 + \Vert\dot{\Sigma}\Vert_F^2 + \Vert \dot{B} \Sigma_1 - \Sigma_1 \dot{B}' \Vert_F^2 \\
+ \Vert \dot{C} \Sigma_2 - \Sigma_2 \dot{C}' \Vert_F^2 + \Vert \dot{E}\Sigma_1 - \Sigma_2 \dot{E}' \Vert_F^2 + \Vert \Sigma_2 \dot{E} - \dot{E}' \Sigma _1 \Vert_F^2 = 1}} 
\sqrt{\Vert \dot{X} S \Vert_F^2 + \Vert \dot{E} \Vert_F^2}.
\end{align}
\end{linenomath}

\subsection{Eliminating variables} \label{sec_sub_eliminating}
Next, we need a lemma, which states that variables that do not feature in the objective in \cref{eqn_sup_complicated} can be eliminated.
\begin{lemma}
Let $\mathrm{W}_j$, $j=1,\ldots,J$, be linear spaces.
Let $L : \mathrm{W}_1 \to \RR^{m_1\times n_1}$ be an invertible linear map and $M : \mathrm{W}_1 \to \RR^{m_1' \times n_1'}$ a nonzero linear map. Let $f_j : \mathrm{W}_2\times\cdots\times \mathrm{W}_J \to \RR^{m_j \times n_j}$ and $g_j: \mathrm{W}_2\times\cdots\times \mathrm{W}_H \to \RR^{m_j' \times n_j'}$ be continuous maps. Let $1 \le H < J$. Then,
\begin{linenomath}
 \begin{multline*}%\label{eqn_sup_simplification}
  \sup_{\substack{X_j \in \mathrm{W}_j,\; j=1,\ldots,J,\\ \Vert L(X_1)\Vert_F^2 + \sum_{j=1}^u \Vert f_j(X_2,\ldots,X_J)\Vert_F^2=1}} \Vert M(X_1) \Vert_F^2 + \sum_{j=1}^v \Vert g_j(X_2,\ldots,X_H) \Vert_F^2\\
  = \sup_{\substack{X_j \in \mathrm{W}_j,\; j=1,\ldots,H,\\ \Vert L(X_1)\Vert_F^2 + \sum_{j=1}^u \Vert f_j(X_2,\ldots,X_H,X_{H+1}^*,\ldots,X_J^*)\Vert_F^2=1}} \hspace{-8pt} \Vert M(X_1) \Vert_F^2 + \sum_{j=1}^v \Vert g_j(X_2,\ldots,X_H) \Vert_F^2,
 \end{multline*}
 \end{linenomath}
where $X_{H+1}^*(X_1,\ldots,X_H), \ldots, X_{J}^*(X_1,\ldots,X_H)$ solves
\begin{linenomath}
\begin{align}\label{eqn_minnorm}
 \min_{\substack{X_j \in \mathrm{W}_j,\\ j=H+1,\ldots,J}}\; \sum_{j=1}^u \Vert f_j(X_2,\ldots,X_H,X_{H+1},\ldots,X_J) \Vert_F^2
\end{align}
\end{linenomath}
for fixed $X_1,\ldots,X_H$, if such minimizers exist.
 \end{lemma}
\begin{proof}
 Let $X_1, \ldots, X_J$ be any feasible point of the constraint set 
 \[
 \mathcal{C}= \left\{ X_j \in \mathrm{W}_j,\; j=1,\ldots,J \;\mid\; \Vert L(X_1)\Vert_F^2 + \sum_{j=1}^u \Vert f_j(X_2,\ldots,X_J)\Vert_F^2=1 \right\}.
 \]
If $X_{H+1},\ldots,X_J$ is not a solution of \cref{eqn_minnorm}, then there exist some $X_{j}' \in \mathrm{W}_j$ for $j=H+1,\ldots,J$ such that
\[
0 \le \Vert L(X_1) \Vert_F^2 + \sum_{j=1}^u \Vert f_j(X_2,\ldots,X_H,X_{H+1}',\ldots,X_J') \Vert_F^2 = \alpha < 1.
\]
We distinguish between two cases.

If $M(X_1)=0$, then we can replace $X_1$ by some $X_1' \not\in \ker(M)$ so that $\Vert L(X_1')\Vert_F^2 = 1-\alpha+\Vert L(X_1)\Vert_F^2$ and $\Vert M(X_1')\Vert_F^2 > 0$. Then $(X_1',X_2,\ldots,X_H,X_{H+1}',\ldots,X_J')\in\mathcal{C}$ and this feasible point has a strictly higher objective value.

Otherwise, $M(X_1) \ne 0$ (so $X_1\ne0$). Let $X_1' = \sqrt{1 + \frac{1-\alpha}{\Vert L(X_1)\Vert_F^2}} X_1$. Then, the point $(X_1',X_2,\ldots,X_H,X_{H+1}',\ldots,X_{J}')\in\mathcal{C}$ is feasible and has a strictly higher objective value because $M$ is linear.

Hence, every feasible point in $\mathcal{C}$ can be replaced by another feasible point in $\mathcal{C}$ wherein $X_{H+1}$, $\ldots$, $X_J$ solves \cref{eqn_minnorm} and whose objective value is not strictly smaller. This concludes the proof.
\end{proof}

Applying this lemma to \cref{eqn_sup_complicated}, treating $\dot{X}$ as the distinguished variable $X_1$, we can eliminate $\dot{Y}$, $\dot{B}$, $\dot{B}'$, $\dot{C}$, $\dot{C}'$, $\dot{\Sigma }$, and $\dot{E}'$. By the lemma, we should plug in their optimal values instead, corresponding to the minimum of 
\begin{linenomath}\small
\begin{multline*}
 \Vert\dot{Y}\Sigma \Vert_F^2 + \Vert\dot{\Sigma}\Vert_F^2 
 + \Vert \dot{B} \Sigma_1 - \Sigma_1 \dot{B}' \Vert_F^2 + \Vert \dot{C} \Sigma_2 - \Sigma_2 \dot{C}' \Vert_F^2 + \Vert \dot{E}\Sigma_1 - \Sigma_2 \dot{E}' \Vert_F^2 + \Vert \Sigma_2 \dot{E} - \dot{E}' \Sigma_1 \Vert_F^2.
\end{multline*}
\end{linenomath}
Because of the sum of squares structure, it is clear that we can take $\dot{Y}^*$, $\dot{B}^*$, $\dot{B}'^*$, $\dot{C}^*$, $\dot{C}'^*$, and $\dot{\Sigma }^*$ all equal to zero in a minimum, irrespective of the values of the other variables. Thus, it only remains to determine, for fixed $\dot{E}$, the solution of 
\[
 \min_{\dot{E}'\in\RR^{r-k\times k}} \Vert \dot{E}\Sigma_1 - \Sigma_2 \dot{E}' \Vert_F^2 + \Vert \Sigma_2 \dot{E} - \dot{E}' \Sigma_1 \Vert_F^2.
\]
By exploiting the basic property $\mathrm{vec}(AXB^T) = (A \otimes B)(\mathrm{vec}(X))$, where the vectorization $\mathrm{vec}$ is in the lexicographic order \cite{Lim2021}, we can rewrite the norms as
\begin{equation}\label{eqn_finalnorms}
 \Vert \dot{E} \Sigma_1 - \Sigma_2 \dot{E}' \Vert_F^2 + \Vert \Sigma_2 \dot{E} - \dot{E}' \Sigma_1 \Vert_F^2 = \left\Vert \begin{bmatrix} I \otimes \Sigma_1 \\ \Sigma_2 \otimes I \end{bmatrix} \dot{\vect{e}} - \begin{bmatrix} \Sigma_2 \otimes I \\ I \otimes \Sigma_1 \end{bmatrix} \dot{\vect{f}} \right\Vert_F^2,
\end{equation}
where $\dot{\vect{e}}=\mathrm{vec}(\dot{E})$ and $\dot{\vect{f}}=\mathrm{vec}(\dot{E}')$.
Minimizing the right hand side, for fixed $\dot{E}$, is a linear least squares problem in $\dot{E}'$. The minimizer is
\begin{linenomath}
\begin{align*}
\dot{\vect{f}} =
 \begin{bmatrix} \Sigma_2 \otimes I \\ I \otimes \Sigma_1 \end{bmatrix}^\dagger \begin{bmatrix} I \otimes \Sigma_1 \\ \Sigma_2 \otimes I \end{bmatrix} \dot{\vect{e}}
 &= (\underbrace{\Sigma_2^2 \otimes I + I \otimes \Sigma_1^2}_T)^{-1} \begin{bmatrix} \Sigma _2 \otimes I \\ I \otimes \Sigma _1 \end{bmatrix}^T \begin{bmatrix} I \otimes \Sigma _1 \\ \Sigma _2 \otimes I \end{bmatrix} \dot{\vect{e}} \\
 &= 2 T^{-1} (\Sigma_2 \otimes \Sigma_1) \dot{\vect{e}} \\
 &= 2 (\Sigma_2 \otimes \Sigma_1) T^{-1} \dot{\vect{e}}.
\end{align*}
\end{linenomath}
In the above, $(A\otimes B)(C \otimes D) = (AC) \otimes (BD)$ for compatible matrices was exploited \cite{Greub1978}. The final step used that diagonal matrices commute.
Observe that
\begin{linenomath}
\begin{align*}
D_1 := (I\otimes \Sigma_1) - 2 (\Sigma_2\otimes I)(\Sigma_2\otimes\Sigma_1)T^{-1}
&= (I \otimes \Sigma_1)(I\otimes \Sigma_1^2 - \Sigma_2^2\otimes I)T^{-1},\\
D_2 := (\Sigma_2 \otimes I) - 2 (I\otimes\Sigma_1)(\Sigma_2\otimes \Sigma_1)T^{-1}
&= (\Sigma_2 \otimes I)(\Sigma_2^2\otimes I - I\otimes\Sigma_1^2) T^{-1},
\end{align*}
\end{linenomath}
so that
\[
 D^2 := D_1^2+D_2^2
 = (\Sigma_2^2 \otimes I - I \otimes \Sigma_1^2)^2 (\Sigma_2^2 \otimes I + I\otimes\Sigma_1^2) T^{-2}
 = (\Sigma_2^2 \otimes I - I \otimes \Sigma_1^2)^2 T^{-1}.
\]
Exploiting that $\left\Vert\left[\begin{smallmatrix}D_1 \\ D_2 \end{smallmatrix}\right]\vect{x} \right\Vert = \Vert \sqrt{D_1^2 + D_2^2} \vect{x} \Vert$ for all diagonal matrices $D_1, D_2$ and vectors $\vect{x}$, we deduce from the foregoing that \cref{eqn_finalnorms} reduces to
\(
 \Vert D \dot{\vect{e}} \Vert_F^2.
\)
In conclusion, if $\dot{E}'(\dot{E})$ is selected optimally, then we have at $A=U\Sigma V^T$ that
\[
 \kappa[\mathcal{L}_\pi^\RR](A) 
 = \sup_{\substack{\dot{X}\in\RR^{m-r\times r},\, \dot{E}\in\RR^{r-k \times k},\\ \Vert\dot{X}\Sigma \Vert_F^2 + \Vert D \mathrm{vec}(\dot{E}) \Vert^2 = 1}}
 \sqrt{\Vert \dot{X} S \Vert_F^2 + \Vert \mathrm{vec}(\dot{E}) \Vert^2}.
\]

\subsection{Computing the condition number} \label{sec_sub_condition}
Since both $\Sigma $ and $D$ are invertible, we
can reparameterize $\dot{X} \mapsto (I \otimes \Sigma^{-1}) \mathrm{vec}(\dot{X})$ and $\mathrm{vec}(\dot{E}) \mapsto D^{-1} \mathrm{vec}(\dot{E})$. This yields
\begin{linenomath}
\begin{align*}
 \kappa[\mathcal{L}_\pi^\RR](A) 
&= \sup_{\substack{(\dot{\vect{x}},\dot{\vect{e}})\in\RR^{d},\; \Vert(\dot{\vect{x}},\dot{\vect{e}})\Vert=1}} \left\Vert \begin{bmatrix} I_{m-r} \otimes S \Sigma^{-1} & 0 \\ 0 &  D^{-1} \end{bmatrix} \begin{bmatrix}\dot{\vect{x}}\\ \dot{\vect{e}} \end{bmatrix} \right\Vert,
\end{align*}
\end{linenomath}
where $d = (m-r)r + (r-k)k$.
Since  $D^{-1} = \sqrt{ \Sigma_2^2\otimes I + I \otimes \Sigma_1^2} \cdot |\Sigma_2^2\otimes I - I\otimes \Sigma_1^2|^{-1}$, where the root and absolute value are applied elementwise to the diagonal entries, and $I_{m-r}\otimes S \Sigma^{-1}$ are diagonal, the spectral norm in the final equality is
\begin{linenomath}
\begin{align}\label{eqn_kappa}
 \kappa[\mathcal{L}_\pi^\RR](A) = \max\left\{ \kappa_{U_\perp},\, \kappa_U \right\},
 \end{align}
 \end{linenomath}
where 
\begin{linenomath}
\begin{align*}
 \kappa_{U_\perp} = 
 \begin{cases} 
 \underset{i \in \pi}{\max}\, \sigma_{i}^{-1} &\text{if } m > r = n\\
 0 & \text{otherwise},
 \end{cases} 
\;\text{and}\;
\quad
\kappa_{U} = \max_{\substack{i \in \{1,\ldots,k\},\\ j \in \{k+1,\ldots,r\}}} \frac{1}{\vert\sigma_i-\sigma_j\vert} \sqrt{\frac{\sigma_i^2 +\sigma_j^2}{(\sigma_i + \sigma_j)^2}} 
\end{align*}
\end{linenomath}
are the spectral norms of $I_{m-r}\otimes S \Sigma^{-1}$ and $D^{-1}$, respectively. 
Observe that characterization in \cref{eqn_kappa} is equivalent to \cref{eqn_cond_expression}, taking into account that $\pi=\{1,\ldots,k\}$ and $\sigma_{r+1}=\cdots=\sigma_m=0$. 
When the maximization is over an empty set, i.e., $\pi=\emptyset$, then $\mathcal{L}_\pi^\RR = 0$ is a constant map, so its condition number is $0$. 
This proves \cref{thm_main} for full-rank real matrices with distinct singular values.

\subsection{Continuous extension to all real matrices} \label{sec_sub_thm1_equal}
The proof of the real case of \cref{thm_main} will be concluded by the uniqueness properties of continuous extensions.

Recall that $\hat{f} : \hat{X} \to \RR\cup\{\pm\infty\}$ is called a \textit{continuous extension} of a continuous function $f : X \to \RR$ to $\hat{X}$ if $\hat{f}$ is continuous and $\hat{f}|_X = f$, where $X \subset \hat{X} \subset \RR^d$. The value of any continuous extension of $f$, \emph{if it exists}, at $x_* \in (\hat{X}\cap\overline{X})\setminus X$, where $\overline{X}$ is the Euclidean closure of $X$, is the unique limit of $f(x_n)$ for (Cauchy) sequences $x_n \in X$ converging to $x_*$ (otherwise continuity of the extension would fail at $x_*$). If $\hat{X} \subset \overline{X}$, then this completely defines $\hat{f}$ pointwise. This proves the following result.

\begin{lemma}\label{lem_context_unique}
 If $\hat{X} \subset \overline{X}$ and there exists a continuous extension of $f : X \to \RR$ to $\hat{X}$, then this extension is unique.
\end{lemma}

We need two further basic results.

\begin{lemma}
$\kappa_\pi^\RR : \mathcal{U}_\pi \to \RR$ from \cref{eqn_cond_expression} is a continuous function.
\end{lemma}
\begin{proof}
The individual singular value functions are
\[
\sigma_i : \kk^{m\times n} \to \RR,\, X \mapsto
\begin{cases}
\sqrt{\lambda_i(X^T X)} & \text{if } m \ge n, \\
\sqrt{\lambda_i(X X^T)} & \text{otherwise},
\end{cases}
\quad\text{for } i=1,\ldots,\min\{m,n\},
\]
where the eigenvalue functions $\lambda_i$ take a symmetric matrix to its $i$th largest eigenvalue. Since the $\lambda_i$ are continuous functions \cite[Chapter 2, Section 6.4]{Kato1976}, the $\sigma_i$'s are continuous. The constants $\sigma_{r+1}=\dots=\sigma_m=0$, if any, are also continuous functions on $\mathcal{U}_\pi$.
\end{proof}

\begin{lemma}
$\kappa[\mathcal{L}_\pi^\RR] : \mathcal{U}_\pi \to \RR$ is a continuous function.
\end{lemma}
\begin{proof}
$\mathcal{L}_\pi^\RR : \mathcal{U}_\pi \to \RR$ is an analytic function \cite[Theorem 6.1 in Chapter 2]{Kato1976}. By equation \cref{eqn_simplified_condition} and the continuity of the spectral norm, this concludes the proof.
\end{proof}

Consider the set $\mathcal{U} \subset \RR^{m \times n}$ of full-rank matrices with distinct singular values. Then, the proof of \cref{thm_main} up to this point has established that 
\[
 \kappa[\mathcal{L}_\pi^\RR]|_{\mathcal{U}} = \kappa_\pi^\RR|_\mathcal{U} := \kappa.
\]
This means that both $\kappa[\mathcal{L}_\pi^\RR]$ and $\kappa_\pi^\RR$ are continuous extensions of $\kappa$. By combining \cref{lem_context_unique} with the next elementary result from algebraic geometry, which implies that the closure of $\mathcal{U}$ is $\RR^{m\times n}$, equality of $\kappa[\mathcal{L}_\pi^\RR]$ and $\kappa_\pi^\RR$ on $\mathcal{U}_\pi$ is established.

\begin{lemma}\label{lem_zariski}
 $\mathcal{U}$ is a nonempty open subset in the Zariski topology on $\RR^{m\times n}$.
\end{lemma}
\begin{proof}
A matrix $A\in\RR^{m\times n}$ has coinciding singular values if and only if the characteristic polynomial $\chi$ of $A^T A \in \RR^{n \times n}$ (if $m \ge n$) or $AA^T \in \RR^{m\times m}$ (if $m \le n$) has at least a double root. This occurs if and only if the discriminant of $\chi$ vanishes \cite[Chapter 12, Section 1.B]{GKZ1994}. Similarly, $A$ has a singular value equal to zero if and only if the determinant of $A^T A$ (if $m \ge n$) or $AA^T$ (if $m \le n$) vanishes.
Hence, $A \not\in \mathcal{U}$ if and only if the determinant times discriminant vanishes.
\end{proof}

For $A \not\in \mathcal{U}_\pi$, the condition number is $\kappa[\mathcal{L}_\pi^\RR] = \infty$ by definition. The formula for $\kappa_\pi^\RR$ in \cref{eqn_cond_expression} is well defined on $\RR^{m\times n}$ and also evaluates to $\infty$ because (i) \cref{rem_bounded} bounds the second factor in \cref{eqn_cond_expression} by constants and (ii) $A\not\in\mathcal{U}_\pi$ requires that there is some $i\in\pi$ and $j\in\pi^c$ such that $\sigma_i=\sigma_j$, where $\sigma_{r+1}=\cdots=\sigma_{m}=0$. Hence, $\kappa_\pi^\RR$ in \cref{eqn_cond_expression} coincides with $\kappa[\mathcal{L}_\pi^\RR]$ on the whole $\RR^{m\times n}$.

This concludes the proof of \cref{thm_main} in the real case.

\section{Proof of \cref{thm_main}: The complex case}
\label{sec_sub_thm1_complex}

The proof of \cref{thm_main,thm_secondary} for $\kk=\CC$ is considerably shorter and relies on the restriction of scalars
\begin{linenomath}
\begin{align*}
 \jmath: \CC^{m\times n} \to \RR^{2m\times 2n},\quad
 A + \imath B \mapsto \begin{bmatrix} A & B \\ -B & A \end{bmatrix},
\end{align*}
\end{linenomath}
where $A, B \in\RR^{m\times n}$ are real matrices, to establish a tight upper bound on $\kappa[\mathcal{L}_\pi^\CC]$ using the real condition number.
Note that $\jmath$ satisfies
\[
\jmath(XY) = \jmath(X)\jmath(Y);\quad
\jmath(X+Y) = \jmath(X)+\jmath(Y);\quad
\jmath(X^H) = \jmath(X)^T
\]
for compatible complex matrices $X$ and $Y$.
For square matrices, we additionally have that a real matrix $R = R + \imath 0$ is mapped to the block diagonal matrix $\jmath(R) = \diag(R,R)$. In particular we have $\jmath(I_m)=I_{2m}$.

If $X\in\CC^{m\times n}$ admits a compact SVD $X = U \Sigma V^H$ where $U$ and $V$ are matrices with unitary columns (i.e., $U^H U = I$ and $V^H V = I$) and $\Sigma$ is a diagonal matrix, then we have that
\(
\jmath(U \Sigma V^H) = \jmath(U) \jmath(\Sigma) \jmath(V)^T
\)
is a \textit{real} compact SVD, up to ordering of the singular values.
Indeed, if $U \in \CC^{m \times n}$ has unitary columns, then $\jmath(U)\in \RR^{2m\times 2n}$ has orthonormal columns because from $U^H U = I$, we conclude
\(
\jmath(U)^T \jmath(U) = \jmath(U^H U) = I.
\)
The multiplicity of each distinct singular value of $A$ is doubled in $\jmath(A)$, however.

If we equip $\CC^{m\times n}$ with the scaled Frobenius norm $\sqrt{2}\|\cdot\|_F$ and $\RR^{m\times n}$ with the usual Frobenius norm $\|\cdot\|_F$, then $\jmath$ is an isometry onto its image. Consequently, we get for $A, A' \in \mathcal{U}_\pi$, with $\pi$ as in \cref{thm_main}, that
\[
\frac{\sqrt{2}\frac{1}{\sqrt{2}}\Vert \mathcal{L}_\pi^\CC(A) - \mathcal{L}_\pi^\CC(A') \Vert_F}{\sqrt{2} \Vert A - A' \Vert_F}
 =\frac{\frac{1}{\sqrt{2}}\Vert \jmath(U)\jmath(S)\jmath(U)^T - \jmath(U') \jmath(S) \jmath(U')^T \Vert_F}{\Vert \jmath(A) - \jmath(A')\Vert_F}.
\]
Taking the limit superior over $A'$ in an infinitesimal neighborhood of $A$ in $\mathcal{U}_\pi$ yields
\[
\kappa[\mathcal{L}_{\pi}^\CC](A) =
\lim_{\epsilon\to0} \hspace{-3pt} \sup_{\substack{\jmath(A')\in\jmath(\mathcal{U}_\pi),\\ \Vert \jmath(A)-\jmath(A')\Vert_F\le\epsilon}} \hspace{-10pt}
 \frac{\Vert \jmath(U)\jmath(S)\jmath(U)^T - \jmath(U') \jmath(S) \jmath(U')^T \Vert_F}{\sqrt{2} \Vert \jmath(A) - \jmath(A')\Vert_F} \\
\le \kappa[\mathcal{L}_{\pi^2}^\RR](\jmath(A)),
\]
where $\pi^2 := (2\pi_1-1, 2\pi_1,2\pi_2-1, 2\pi_2,\ldots,2\pi_k-1, 2\pi_k)$, which originates from $\jmath(S)=\left[\begin{smallmatrix}S \\ & S \end{smallmatrix}\right]$ and assumes the singular values of $\jmath(A)$ were also sorted decreasingly. 
The upper bound exploited that the perturbations of $\mathcal{L}_\pi^\CC$ are restricted to the linear subspace $\jmath(\CC^{m\times n}) \subsetneq \RR^{2m\times 2n}$, while the real condition number at $\jmath(A)$ allows arbitrary perturbations in $\RR^{2m\times 2n}$.
Since \cref{thm_main} was already proved in the real case, $\kappa[\mathcal{L}_{\pi^2}^\RR](\jmath(A))$ is given by the formula \cref{eqn_cond_expression}, with $\pi^2$ as selected singular values. In fact, by exploiting the structure of $\pi^2$, which selects the same singular values as $\pi$ including the duplicate that is introduced by $\jmath$, it can be verified that 
\begin{equation*}
\kappa[\mathcal{L}_{\pi}^\CC](A)\le
\kappa[\mathcal{L}_{\pi^2}^\RR](\jmath(A)) = \max_{\substack{i\in\pi,\\ j\in\pi^c}} \frac{1}{\vert \sigma_i(A) - \sigma_j(A) \vert} \sqrt{\frac{\sigma_i^2(A) + \sigma_j^2(A)}{(\sigma_i(A)+\sigma_j(A))^2}},
\end{equation*}
where $\sigma_i(A)$ is the $i$th singular value of $A$ and $\pi^c = \{1,\ldots,m\}\setminus\pi$.
The right-hand side is exactly \cref{eqn_cond_expression} for $A\in\CC^{m\times n}$.

\Cref{sec_sub_worstdirection} will prove that there exists a perturbation attaining the foregoing upper bound.
This will conclude the proof of \cref{thm_main} for $\kk=\CC$.

\section{Proof of \cref{thm_secondary}}\label{sec_sub_worstdirection}
Let $A=U\Sigma V^H \in \kk^{m\times n}$ be any full SVD of a rank-$r$ matrix with singular values $\sigma_1 \ge \cdots \ge \sigma_r > \sigma_{r+1} = \cdots = \sigma_m = 0$.
Let $i$ and $j$ be as in the statement of \cref{thm_secondary} and consider the $m\times m$ and $n\times n$ matrices
\begin{linenomath}
\begin{equation*}\label{eqn_worst_inside}
\dot{U} = U (\vect{e}_j \vect{e}_i^T - \vect{e}_i \vect{e}_j^T)
\quad\text{ and }\quad
\dot{V} = 
\begin{cases} 
2 \frac{\sigma_i \sigma_j}{\sigma_i^2 + \sigma_j^2} V(\vect{e}_j \vect{e}_i^T - \vect{e}_i \vect{e}_j^T), & \text{if } 1 \le i,j \le r,\\
0, & \text{otherwise}.
\end{cases}
\end{equation*}
\end{linenomath}
Recall from \cite{EAS1998,Manton2002} that $\dot{U}$ is a tangent vector to the group
\[
\mathcal{O}_m^\kk =\{ U \in \kk^{m\times m} \mid U^H U = U U^H = I \}
\]
of orthogonal ($\kk=\RR$) or unitary ($\kk=\CC$) matrices. The analogous statement holds for $\dot{V}$.
In particular, we have
\[
 U_t = U + t \dot{U} + o(t) \in \mathcal{O}_{m}^\kk \quad\text{and}\quad
 V_t = V + t \dot{V} + o(t) \in \mathcal{O}_{n}^\kk,
\]
for sufficiently small $t > 0$. Then, we have
\[
A_t 
:= U_t \Sigma V_t^H 
= (U + t\dot{U})\Sigma (V + t\dot{V})^H + o(t) 
= A + t \dot{A} + o(t),
\]
where $\dot{A}$ is as in the statement of \cref{thm_secondary}.

On the one hand, we have
\begin{linenomath}
\begin{align}
 d_{\Gr}^c( \mathcal{L}_\pi^\kk(A), \mathcal{L}_\pi^\kk(A_t))
 \nonumber&= \frac{1}{\sqrt{2}} \| U S U^H - U_t S U_t^H \|_F \\
 \nonumber&= \frac{t}{\sqrt{2}} \| (\vect{e}_j \vect{e}_i^T - \vect{e}_i \vect{e}_j^T) S + S (\vect{e}_j \vect{e}_i^T - \vect{e}_i \vect{e}_j^T)^T \|_F + o(t) \\
 \label{eqn_worst_distgr}&= t + o(t),
\end{align}
\end{linenomath}
having used in the final step that $S$ is a projector for which $S \vect{e}_i = \vect{e}_i$ and $S \vect{e}_j = 0 \vect{e}_j$.
On the other hand, if $1\le i,j\le r$, then we find after some basic computations that essentially repeat the end of \cref{sec_sub_eliminating}, that
\begin{linenomath}
\begin{align}
 \| A_t - A \|_F
 \nonumber&= \| U_t \Sigma V_t^H - U \Sigma V^H \|_F \\
 \nonumber&= t\left\| (\vect{e}_j \vect{e}_i^T - \vect{e}_i \vect{e}_j^T) \Sigma + 2 \frac{\sigma_i \sigma_j}{\sigma_i^2 + \sigma_j^2} \Sigma  (\vect{e}_i \vect{e}_j^T - \vect{e}_j \vect{e}_i^T) \right\|_F + o(t) \\
 \label{eqn_worst_disteucl}&= t |\sigma_i-\sigma_j| \sqrt{\frac{(\sigma_i+\sigma_j)^2}{\sigma_i^2+\sigma_j^2}} + o(t).
\end{align}
\end{linenomath}
If $i > r$, then $\Vert A_t - A \Vert_F = t \Vert (\vect{e}_j \vect{e}_i^T - \vect{e}_i \vect{e}_j^T) \Sigma \Vert_F + o(t) = t \sigma_j + o(t)$, having used that $\sigma_i = 0$.
Note that $\sigma_j = |\sigma_i-\sigma_j| |\sigma_i+\sigma_j| (\sigma_i^2+\sigma_j^2)^{-1/2}$ in this case as well, even when $\sigma_j=0$ because of \cref{rem_bounded}.
The analogous result holds for $j > r$.
This shows that \cref{eqn_worst_disteucl} applies for all $1 \le i, j \le m$.

Plugging \cref{eqn_worst_distgr,eqn_worst_disteucl} into the middle formula in \cref{eqn_simplified_condition} results in \cref{eqn_cond_expression}.
This proves \cref{thm_secondary}.

\providecommand{\bysame}{\leavevmode\hbox to3em{\hrulefill}\thinspace}
\providecommand{\MR}{\relax\ifhmode\unskip\space\fi MR }
\providecommand{\MRhref}[2]{%
  \href{http://www.ams.org/mathscinet-getitem?mr=#1}{#2}
}
\providecommand{\href}[2]{#2}

\end{document}